\documentclass{amsart}
\usepackage[foot]{amsaddr}
\usepackage{amssymb}
\usepackage{amsthm}
\usepackage{mathtools}
\usepackage{environ}
\usepackage{enumerate}
\usepackage{microtype}
\usepackage{enumitem}
\usepackage[utf8]{inputenc}

\theoremstyle{plain}
\newtheorem{theorem}{Theorem}[section]
\newtheorem{lemma}[theorem]{Lemma}
\newtheorem{proposition}[theorem]{Proposition}
\newtheorem{corollary}[theorem]{Corollary}

\theoremstyle{definition}
\newtheorem{definition}[theorem]{Definition}

\theoremstyle{remark}
\newtheorem*{remark}{Remark}

\newcommand{\bd}{\partial}
\newcommand{\eps}{\varepsilon}

\newcommand{\psh}{\mathcal{PSH}}
\newcommand{\usc}{\mathcal{USC}}
\newcommand{\lsc}{\mathcal{LSC}}
\newcommand{\suchthat}{\mathrel{;}}

\DeclareMathOperator\supp{supp}

\DeclareMathOperator{\dist}{dist}

\DeclareMathOperator{\ddc}{dd^c}

\DeclareMathOperator{\hull}{hull}
\DeclareMathOperator{\relcomp}{\subset\!\subset}

\DeclareMathOperator{\lepsh}{\le_\mathrm{psh}}
\DeclareMathOperator{\lepshc}{\le_\mathrm{c}}
\DeclareMathOperator{\lepshm}{\le_*}
\DeclareMathOperator{\leF}{\le_{\mathcal{F}}}

\newcommand{\ext}{\mathrm{d}}

\newcommand{\plsh}{\mathcal{PSH}}
\newcommand{\upsc}{\mathcal{USC}}

\newcommand{\Cn}{\mathbb{C}^{n}}

\newcommand{\R}{\mathbb{R}}

\newcommand{\E}{\mathcal{E}}

\newcommand{\M}{\mathcal{M}}

\numberwithin{equation}{section}

\title[Variations on a theorem by Edwards]{Variations on a theorem by Edwards}
\author{Mårten Nilsson, Frank Wikström}
\address{Center for Mathematical Sciences\\
  Lund University\\
  Box 118, SE-221 00 Lund, Sweden}
\email{marten.nilsson@math.lth.se}
\email{frank.wikstrom@math.lth.se}


\begin{document}

\maketitle
\begin{abstract}
    We discuss two variations of Edwards' duality theorem. More precisely, we
    prove one version of the theorem for cones not necessarily containing all
    constant functions. In particular, we allow the functions in the cone to
    have a non-empty common zero set. In the second variation, we replace
    suprema of point evaluations and infima over Jensen measures by suprema of
    other continuous functionals and infima over a set measures defined through
    a natural order relation induced by the cone.

    As applications, we give some results on propagation of discontinuities
    for Perron--Bremermann envelopes in hyperconvex domains as well as a characterization of minimal
    elements in the order relation mentioned above.
\end{abstract}

\section{Introduction}

Let $X$ be a metric space, and let $\mathcal{F}$ be a cone of upper
semicontinuous functions on~$X$. We say that a compactly supported, positive
measure $\mu$ is a \textit{Jensen measure} for $\mathcal{F}$ with barycenter $x
\in X$, denoted $\mu \in J_x^{\mathcal{F}}$, if
\[
u(x)\leq \int u \,d\mu, \quad\text{for all } u \in \mathcal{F}.
\]
In many situations, it is possible to retrieve interesting information about~$X$
and~$\mathcal F$ from the sets~$J_x^{\mathcal{F}}$. For example, the notions of
hyperconvexity and B-regularity in pluripotential theory may be characterized in
terms of Jensen measures of cones of plurisubharmonic functions~\cite{Wikstrom}.
Another application of Jensen measures to pluripotential theory is in the study
of pointwise suprema over elements in~$\mathcal{F}$~\cite{cole,
nilssonwik,nilsson, poletsky, Wikstrom}. A general result is the following
theorem due to Edwards~\cite{edwards}:

\begin{theorem}[Edwards' theorem]
    Suppose that~$X$ is compact, that~$\mathcal{F}$ contains all constant
    functions, and that $g:X \to (-\infty, \infty]$ is lower
    semicontinuous. Then
    \[
        \sup\big\{u(x) \suchthat u\in \mathcal{F}, u \leq g\big\} =
        \inf\big\{\int g \, d\mu \suchthat \mu \in J_x^{\mathcal{F}}\big\}.
    \]
\end{theorem}

In this note, we prove two variations of Edwards' theorem. In Section~2, we
relax the conditions on~$\mathcal{F}$ in a manner that allows us to apply the
theorem to cones of negative plurisubharmonic functions in hyperconvex domains,
paralleling the case when the domain is B-regular~\cite{nilsson, Wikstrom}. This
is done by modifying the orginal proof of Edwards', replacing the space $C(X)$
by  a subspace $H \subset C(X)$ satisfying certain criteria. We then use this
result to establish continuity of certain Perron--Bremermann envelopes.

Section~3 and Section~4 provide the setting for the second variation of Edwards'
theorem, where the Jensen measures are replaced by a set of measures given by an
ordering induced by~$\mathcal{F}$. Here, we are motivated by an ordering of
positive measures induced by (negative) plurisubharmonic functions, introduced
by Bengtson in~\cite{Bengtson}. More precisely,

\begin{definition}
    Let $\Omega \subset  \Cn$ be a domain, and let $\mu$ and $\nu$ be
    positive Radon measures supported on $\bar\Omega$. We say that
    $\mu \lepsh \nu$ if
    \[
        \int u\,\ext\mu \ge \int u\,\ext\nu
    \]
    for all $u \in \E_0(\Omega)$. Recall that
    \[
        \E_0(\Omega) = \{ u \in \plsh(\Omega) : u < 0,
        \lim_{\zeta \to p \in \partial\Omega} u(\zeta) = 0,
        \int_\Omega (\ddc u)^n < \infty \},
    \]
    is the set of ``plurisubharmonic test functions'' introduced by
    Cegrell~\cite{Cegrell}.
\end{definition}

Note that if~$\mu$ is a Jensen measure with barycenter~$z$, then $\mu \lepsh
\delta_z$ (where~$\delta_z$ denotes a point mass at~$z$). The fact that
$\frac12\mu \lepsh \mu$ in Bengtson's ordering gives rise to a number of
normalization issues, for example in the definition of maximal and/or minimal
measures in~\cite{Bengtson}. Inspired by these observations, we introduce a few
(slightly) different notions of orderings induced by plurisubharmonic functions
in Section~3, in which comparable measures have the same total mass. In
Section~4, we continue our study, employing a result of Duval and
Sibony~\cite{Duval:95} which provides sufficient criteria for two comparable
measures to coincide.

Finally, in Section~5, we prove a version of Edwards' theorem adapted to this
setting by letting a fixed measure~$\mu$ play the rôle of~$\delta_z$, and
letting measures plurisubharmonically larger than~$\mu$ play the part of the
Jensen measures. We then apply the result to characterize the minimal elements
in the different orderings.

\section{Duality on hyperconvex domains}

In this section, we study envelopes over cones of negative plurisubharmonic
functions on the closure of a bounded hyperconvex domain~$\Omega$. Specifically,
we consider the two cones
\begin{align*}
    &\psh^0(\bar \Omega) :=
        \{u \in \usc(\bar \Omega)\cap\psh^-(\Omega) \suchthat u\big|_{\bd \Omega} = 0\},   \\
    &\E_0(\Omega) \cap C(\bar \Omega).
\end{align*}
By an approximation theorem of Cegrell~\cite{Cegrell2}, the Jensen measures of
these two cones coincide, using the monotone convergence theorem. Denote this
set of measures by~$J^-_z$. We first aim to prove the following:

\begin{theorem}\label{thm:edwhypcon2}
    Let $g: \bar \Omega \to \R$ be a lower bounded function such
    that $g_* \leq g$, with equality outside a pluripolar set $P$, and suppose that
$g_*(\zeta)=0$ for all $\zeta \in \bd \Omega$. Then
    \begin{align*}
        \sup\big\{u(x) \suchthat u\in \E_0(\Omega) \cap C(\bar \Omega), u \leq g\big\}
        &= \sup\big\{u(x) \suchthat u\in \psh^0(\Omega) , u \leq g\big\} \\
        &= \inf\big\{\int g \, d\mu \suchthat \mu \in J_z^-\}
    \end{align*}
    holds outside the pluripolar hull of~$P$. In particular, if $g$ is upper
    semicontinuous and~$P$ is a closed, complete pluripolar set, then the
    envelopes are continuous outside~$P$.
\end{theorem}

Before embarking on a proof, we begin by establishing a version of Edwards'
theorem adapted to also work for cones~$\mathcal{F}$ where the common zero set
\[
    Z_{\mathcal{F}}:=\{x \in X \suchthat \forall u \in \mathcal{F}, u(x)=0\}
\]
is non-empty. Specifically, we will modify the aspects of the original proof
that rely on the following properties of $C(X)$:
\begin{itemize}
    \item For each element $u \in \usc(X)$, there exists a decreasing sequence $g_i \in C(X)$ such that $g_i \searrow u$. In particular, this holds if $u \in \mathcal{F}$.
    \item For each element $l \in \lsc(X)$, there exists an increasing sequence $g_i \in C(X)$ such that $g_i \nearrow l$.
\end{itemize}
Instead, this rôle will be played by a subspace $H \subset C(X)$, for which the set
\[
    Z_{H}=\{x \in X \suchthat \forall h \in H, h(x)=0\}
\]
is allowed to be non-empty as well. In order to do so, we make the following
definitions.

\begin{definition}
  Let $H \subset C(X)$ be a vector space, and let $E: C_c( X \setminus Z_{H})
  \to C(X)$ denote extension by zero. We say that $H$ is \textit{total
  outside its zeros} if for each $\varphi \in E(C_c(X \setminus Z_{H}))$,
  there exists a decreasing sequence $h_i \in H$ such that
  \(
    h_i \searrow \varphi.
  \)
\end{definition}

\begin{definition}
    Let $H \subset C(X)$ be a vector space. We say that $\mathcal{F}$ is
    $H$\textit{-admissible} if each element in $H$ has minorant in
    $\mathcal{F}$, and for each element $u \in \mathcal{F}$, there exists a
    decreasing sequence $h_i \in H$ such that $h_i \searrow u$.
\end{definition}

\begin{remark}
    Note that a cone of upper semicontinuous functions containing constants is
    $C(X)$-admissible.
\end{remark}

For brevity of notation, we define the following two operators on the
space of Borel functions on $X$:
\begin{align*}
    S_x(g) &=
    \begin{cases}
        \sup\{u(x) \suchthat u\in \mathcal{F}, u \leq g\}, & \hfill \text{if }\exists u\in \mathcal{F} \suchthat u \leq g, \\
        - \infty, &\hfill \text{otherwise,}
    \end{cases} \\
     I_x(g) &= \inf\Big\{\int g \, d\mu \suchthat \mu \in J_x^{\mathcal{F}}\Big\}.
\end{align*}
Adapting Edwards' original proof, we get the following version of the theorem.

\begin{theorem}[Edwards' theorem on $H$-admissible cones]\label{edwards}
    Let $X$ be compact, $H \subset C(X)$ be a subspace total outside its zeros,
    and suppose that $\mathcal{F}$ is $H$-admissible. Fix $x\in X$ and assume
    that there exists $C_x>0$ such that for each element $h \in H$, there exists
    $u\in \mathcal{F}$ with $u\leq h$ and
    \[
        u(x) \geq -C_x|h|_\infty.
    \]
    Then for each non-decreasing sequence $H \ni h_i \nearrow h_\infty$, we have $S_x(h_\infty)=I_x(h_\infty)$.
\end{theorem}

\begin{proof}
    We first consider the case when $h_\infty=h \in H$. Since $h$ is bounded
    from below by an element $u\in\mathcal{F}$ such that $u(x)\geq
    -C_x|h|_\infty$ at $x\in X$, we have
    \[
        \infty > h(x) \geq S_x(h) \geq u(x) > -\infty,
    \]
    so $-S_x$ defines a sublinear map $H\to\mathbb{R}$. On the subspace
    of generated by $h$, we now define a linear map
    \[
        -A(ch)=-cS_x(h), \quad c \in \mathbb{R},
    \]
    which satisfies $-A\leq-S_x$ on
    the whole span of $h$. Using the Hahn--Banach theorem, we first extend $-A$
    to a linear map $-A'$ on $H$ such that
    \[
        -A' \leq -S_x \text{ on }H,
    \]
    and then extend $A'$ to a continuous linear functional $A''$ on $C(X)$ such
    that
    \[
        A''(\varphi) \leq C_x|\varphi|_\infty \text{ for all }\varphi \in C(X).
    \]
    Since $X$ is compact, the Riesz representation theorem implies that~$A''$
    may be represented by a signed measure $\mu_h = \mu_h^+ - \mu_h^-$. Using
    the fact that~$H$ is total outside its zeros, we now note that
    $A''|_{E(C_c(X \setminus Z_{H}))}$ is positive, which implies that~$\mu_h^-$
    is supported on the closed set of common zeros of $H$. For any $g \in H$,
    $\mu_h^+$ hence satisfies
    \[
        \int_X g\,  d\mu_h^+ = \int_X g\,  d\mu_h \geq S_x(g),
    \]
    with equality if $g=h$. Since $\mathcal{F}$ is $H$-admissible, the monotone
    convergence theorem implies that
    \[
        \int_X u\,  d\mu_h \geq u(x)
    \]
    for any $u \in \mathcal{F}$, and so $\mu_h^+$ is a Jensen measure. We
    conclude that $S_x(h)=I_x(h)$.

    For general~$h_\infty$, associate to each~$h_i$ a signed measure~$\mu_{h_i}$ as
    above, and note that
    \[
        C_x \geq \int_X d\mu_{h_i} \geq -C_x.
    \]
    By the Banach--Alaoglu theorem, we may thus extract a subsequence
    $\{\mu_{h_{i_n}}\}$ converging to a signed measure $\mu_{h_\infty}$ in the
    weak$^*$-topology. Fixing $k\in \mathbb{N}$, we have
    \begin{align*}
        I_x(h_\infty)
        &\geq S_x(h_\infty)
        \geq \lim_{n\to \infty} S_x(h_{i_n})
        = \lim_{n\to \infty} I_x(h_{i_n}) \\
        &=\lim_{n \to \infty}\int h_{i_n} \, d\mu_{h_{i_n}}^+
            =\lim_{n \to \infty}\int h_{i_n} \, d\mu_{h_{i_n}} \\
            &\geq \lim_{n\to \infty}\int h_k \, d\mu_{g_{i_n}}
            = \int h_k \, d\mu_{h_\infty}=\int h_k \, d\mu_{h_\infty}^+,
    \end{align*}
    and letting $k \to \infty$ and applying the monotone convergence
    theorem, we get $S_x(h_\infty)=I_x(h_\infty)$.
\end{proof}

\begin{remark}
    In Section~5, we will introduce another layer of generality by replacing
    point evaluations with integration against any positive Radon measure.
\end{remark}

Directly applying Theorem~\ref{edwards}, we get the following lemma.

\begin{lemma}\label{thm:edwhypcon}
    Suppose that $g: \bar \Omega \to \R$ is a lower semicontinuous
    function such that
    $g(\zeta)=0$ for all $\zeta \in \bd \Omega$. Then
    \begin{align*}
        \MoveEqLeft\sup\big\{u(x) \suchthat u\in \E_0(\Omega) \cap C(\bar \Omega), u \leq g\big\} \\
        &= \sup\big\{u(x) \suchthat u\in \psh^0(\bar \Omega) , u \leq g\big\} \\
        &= \inf\Big\{\int g \, d\mu \suchthat \mu \in J_z^- \Big\}.
    \end{align*}
\end{lemma}

\begin{proof}
    We first check that the conditions required for Theorem~\ref{edwards}
    are fulfilled with
    \[
        X= \bar \Omega, \quad
        H= \{ h \in C(\bar \Omega) \suchthat h\big|_{\bd \Omega} =0 \}, \quad
        \mathcal{F}\in \{\psh^0(\Omega), \E_0(\Omega) \cap C(\bar \Omega)\}.
    \]
    Clearly, $H$ is total outside its zeros, and $\mathcal{F}$ is
    $H$-admissible. For each $h \in H$, we may find $u\in \mathcal{F}$ such that
    $u\leq h$ and $u(z) \geq -|h|_\infty$. This follows for
    $\mathcal{F}=\E_0(\Omega) \cap C(\bar \Omega)$ since
    \[
        u \in \E_0(\Omega) \cap C(\bar \Omega) \implies
            \max \{u, -|h|_\infty\} \in \E_0(\Omega) \cap C(\bar \Omega)
    \]
    by \cite[Lemma~3.4]{Cegrell}. It remains to show that~$g$ may be
    approximated from below by elements in $H$. Pick $\varphi \in H$ such that
    $\varphi \leq g$ (such a function exists by the Katětov--Tong insertion
    theorem~\cite{tong}) and let $g_i \in C(\bar \Omega)$ approximate $g$
    pointwise from below. Then $\varphi_i :=\max\{\varphi, g_i\} \in H$, and
    $\varphi_i \nearrow g$ on $\bar \Omega$.
\end{proof}

\begin{proof}[Proof of Theorem~\ref{thm:edwhypcon2}]
    Pick any point $z_0$ outside the pluripolar hull of $P$ and let $u \in
    \psh^0(\Omega)$ satisfy $u \big|_P = -\infty$, with $u(z_0)>-\infty$. Then
    for any $\mu \in J_{z_0}^-$,
    \[
     \int_P u \, d\mu \geq \int_{\bar \Omega} u \, d\mu \geq u(z_0) > -\infty,
    \]
    and so $\mu(P)=0$. We conclude that $I_{z_0}(g)= I_{z_0}(g_*)$. On the other
    hand, it is clear that $S_{z_0}(g)\leq I_{z_0}(g)$, which together with
    Lemma~\ref{thm:edwhypcon} implies that
    \[
        S_{z_0}(g)\leq I_{z_0}(g) =  I_{z_0}(g_*) =  S_{z_0}(g_*) \leq S_{z_0}(g)
    \]
    for both cones. Since~$z_0$ was an arbitrary point outside the pluripolar
    hull of~$P$, this proves the first statement of the theorem.

    Now suppose that $g$ is upper semicontinuous, and that~$P$ is a closed,
    complete pluripolar set. Then, we have an equality
    \[
        \sup\big\{u(x) \suchthat u\in \E_0(\Omega) \cap C(\bar \Omega), u \leq g\big\} =
        \sup\big\{u(x) \suchthat u\in \psh^0(\Omega) , u \leq g\big\}
    \]
    between a lower semicontinuous function and upper semicontinuous function on
    an open set, which implies that the envelopes must be continuous on that
    set.
\end{proof}

In certain situations, it is possible to guarantee that even non-pluripolar
discontinuities do not propagate; the extremal function of interior balls in
hyperconvex domains are for example continuous everywhere~\cite{kerzman}. The
proof of this fact can be extended to the following gluing construction.

\begin{theorem}
    Let $\Omega$ be a B-regular domain, and let $\Omega' \subset \Omega$ be a
    relatively compact subdomain with the property that for each $p \in \bd
    \Omega'$, there exists a ball $B \subset \Omega'$ with $p \in \bd B$. Let $h
    \in C (\bar \Omega)$, $u \in \psh(\Omega)$ be such that $u \leq h$, and
    assume that $u$ is continuous on $\bar \Omega'$. Then the Perron--Bremermann
    envelope $P(h_{u, \Omega'})$ of
    \[
        h_{u, \Omega'}(z) :=
        \begin{cases}
            u(z) \text{ if }z\in \bar{\Omega'}, \\
            h(z) \text{ if }z\in \Omega \setminus \bar{\Omega'}
        \end{cases}
    \]
    is continuous.
\end{theorem}

    \begin{proof}
    Since $h_{u, \Omega'}$ is upper semicontinuous, $P(h_{u, \Omega'})$ is a
    plurisubharmonic function. Notice that
    \begin{align*}
        P(h_{u, \Omega'})\big|_{\bd \Omega} &= h, \\
        P(h_{u, \Omega'})\big|_{\Omega'} &= u,
    \end{align*}
    and that $P(h_{u, \Omega'})_*\big|_{\bar{\Omega'}}= u$ since $P(h_{u,
    \Omega'})\geq u$ and $u$ is continuous on $\bar \Omega'$. In order to show
    that $P(h_{u, \Omega'})$ is continuous outside $\Omega'$ as well, we will now
    show that
    \[
        P(h_{u, \Omega'})^*\big|_{\bar{\Omega'}}= u.
    \]
    Fix $\varepsilon>0$ and pick $p \in \bd \Omega'$ and two concentric balls
    $B(\zeta, r) \subset B(\zeta, R) \subset \Omega'$ such that
    \[
        p\in \bd B(\zeta, r), \quad |u(z) -u(p)| <
            \varepsilon \text{ on }\bar B(\zeta, r).
    \]
    Then the Perron--Bremermann envelope
    \[
        \varphi_{r,R}(z) := \sup \{v(z) \suchthat v \in \psh^-(B(0, R)), v \big|_{B(0,r)}
        \leq u(p)+\varepsilon - \max_{\bd B(\zeta,R)} h\}
    \]
    is toric, and thus continuous. Hence
    \[
        \limsup_{z \to p} P(h_{u, \Omega'}) \leq
        \limsup_{z \to p} \varphi_{r,R}(z+\zeta) + \max_{\bd B(\zeta,R)} h
        \leq u(p) + \varepsilon.
    \]
    Letting $\varepsilon \to 0$, we conclude that $P(h_{u, \Omega'})$ is
    continuous on the boundary of $\Omega \setminus \bd \Omega'$, and by theorem
    of Walsh~\cite{walsh}, continuous in the interior as well.
\end{proof}

\begin{remark}
    If $h\big|_{\bd \Omega} =0$, it is enough to assume that $\Omega$ is
    hyperconvex.
\end{remark}

Finally, Theorem~\ref{edwards} implies the following statement for envelopes of
a class of unbounded functions.

\begin{theorem}\label{obegransad} Let $\Omega$ be a hyperconvex domain, and let
    $g: \bar \Omega \to [-\infty, 0]$ be such that $g^*=g_*$ on $\bar \Omega$,
    $g = 0$ on $\partial \Omega$. Furthermore, assume that there exists $v \in
    \psh^-(\Omega)$ such that
    \begin{align*}
      & g(z_0)=-\infty \implies v(z_0)=-\infty,\quad\text{and} \\
      & \frac{v(z)}{g(z)} \to \infty \quad\text{as}\quad g(z) \to -\infty.
    \end{align*}
    Then the envelope
    \[
         \sup \big\{u(z) \suchthat u \in \psh^-(\Omega), u^* \le g \big\}
    \]
    is continuous on $\{z \in \bar\Omega \suchthat v_*(z) \neq -\infty\}$.
\end{theorem}

\begin{proof}
    The proof is completely analogous to the case when the domain is B-regular. See \cite[Section~4]{nilsson} for more details.
\end{proof}

\section{Plurisubharmonic orderings of measures}

Let us begin this section by giving two alternative orderings induced by
plurisubharmonic functions.

\begin{definition}
    Let $\Omega \subset \Cn$ be a bounded domain, and let $\M(\bar\Omega)$ denote
    the set of finite positive Radon measures supported on~$\bar\Omega$.

    Let $\mu, \nu \in \M(\bar\Omega)$. We say that $\mu \lepshc \nu$ if
    \[
        \int u\,\ext\mu \ge \int u\,\ext\nu \quad\text{for all $u \in \plsh(\Omega)
        \cap C(\bar\Omega)$}.
    \]
    Note the reversal of the inequality. We choose this notation to keep
    comparisons with Bengtson's ordering easier to formulate. Also note that we
    allow the measures to have positive mass on $\partial\Omega$.

    Similarly, we say that $\mu \lepshm \nu$ if
    \[
        \int u^*\,\ext\mu \ge \int u^*\,\ext\nu \quad\text{for all upper
        bounded $u \in \plsh(\Omega)$}.
    \]
\end{definition}

By requiring the defining inequality to hold, not only for negative
plurisubarmonic functions, comparable measures automatically have the same total
mass.

\begin{proposition}\label{prop:equal_mass}
  If $\mu \lepshc \nu$ (or $\mu \lepshm \nu$), then $\mu(\bar\Omega)
  = \nu(\bar\Omega)$.
\end{proposition}

\begin{proof}
  Apply the definition to the plurisubharmonic functions $u(z) \equiv 1$ and
  $u(z) \equiv -1$, respectively.
\end{proof}

Furthermore, for finite measures on $\Omega$ with equal total mass, Bengtson's
ordering coincides with our ordering.

\begin{proposition}
  If $\mu \lepsh \nu$ (in the sense of Bengtson) and $\mu(\Omega) = \nu(\Omega)
  < \infty$, then $\mu \lepshm \nu$, and hence $\mu \lepshc \nu$.
\end{proposition}

\begin{proof}
  Assume that $\mu \lepsh \nu$. Then, in particular, the measures are carried by
  $\Omega$. By Bengtson~\cite[Prop. 3.2c]{Bengtson}, it follows that
  \[
      \int u\,d\mu \ge \int u\,d\nu
  \]
  for all negative plurisubharmonic function $u$, not just functions in $\E_0$.
  If $v \in \plsh(\Omega)$ is upper bounded, then $u = v-M < 0$ for $M$ large
  enough, and since $\mu(\Omega) = \nu(\Omega)$, it follows that $\int v\,d\mu
  \ge \int v\,d\nu$.
\end{proof}

Let us also remark that there is no difference between the orderings $\lepshc$
and $\lepshm$ on B-regular domains.

\begin{proposition}
  Assume that $\Omega$ is B-regular, and that $\mu$ and $\nu$ are positive Radon
  measures on $\bar\Omega$. Then $\mu \lepshc \nu$ if and only if $\mu \lepshm
  \nu$.
\end{proposition}

\begin{proof}
  Clearly, if $\mu \lepshm \nu$, then $\mu \lepshc \nu$. Conversely, assume that
  $\mu \lepshc \nu$ and let $u \in \plsh(\Omega)$ be an upper bounded
  plurisubharmonic function. Then, by an approximation theorem
  in~\cite{Wikstrom}, there is a decreasing sequence $u_j \in \plsh(\Omega) \cap
  C(\bar\Omega)$, such that $u_j \searrow u^*$ on $\bar\Omega$. Hence, by
  monotone convergence
  \[
      \int u^*\,d\nu = \lim_{j \to \infty} \int u_j\,d\nu
      \le \lim_{j \to \infty} \int u_j\,d\mu = \int u^*\,d\mu,
  \]
  and consequently $\mu \lepshm \nu$.
\end{proof}

On B-regular domains, two measures supported on the boundary are comparable only
if they are equal.

\begin{proposition}\label{prop:boundary_measures}
    Let $\Omega$ be a B-regular
    domain. If $\mu \lepshc \nu$ are both supported on $\partial\Omega$, then
    $\mu = \nu$.
\end{proposition}

\begin{proof}
  Let $\phi \in C(\partial\Omega)$. Since $\Omega$ is B-regular, $\phi$ extends
  to $\plsh(\Omega) \cap C(\bar\Omega)$, and hence
  \[
      \int \phi\,d\nu \le \int \phi\,d\mu
  \]
  for every $\phi \in C(\partial\Omega)$. In other words, $\sigma = \mu - \nu$
  defines a positive measure with support on $\partial\Omega$. By
  Proposition~\ref{prop:equal_mass}, $\sigma(\partial\Omega) = 0$, so $\mu =
  \nu$.
\end{proof}

\begin{remark}
    The same is true for $\lepshm$, even if $\Omega$ is not B-regular, if we
    define the ordering via the cone $\mathcal{F} = \upsc(\bar\Omega) \cap
    \plsh(\Omega)$. Every continuous function on $\partial\Omega$ extends to
    an element in $\mathcal{F}$.

    We also note that if $\Omega$ is hyperconvex, but not B-regular, the
    conclusion in Proposition~\ref{prop:boundary_measures} does not hold. In
    that case, there is a point $p \in \partial\Omega$ and a Jensen measure $\mu
    \in J_p^c$ with $\supp \mu \subset \partial\Omega$, $\mu \neq \delta_p$,
    see~\cite{Wikstrom} for details. Then $\mu \lepshc \delta_p$.
\end{remark}

\section{Hulls and positive currents}\label{sec:hulls}

In a seminal paper by Duval and Sibony~\cite{Duval:95}, it is shown that
polynomial hulls can be described by the support of positive currents of
bidimension (1,1). In fact, several of their results may be localized to study
hulls with respect to other classes of plurisubharmonic functions. (Recall that
the polynomial hull is the same as the hull with respect to $\plsh(\Cn)$.)

To be more precise, let $\Omega$ be a bounded domain in $\Cn$ and let $\E' =
\E'(\Omega)$ be the space of distributions with compact support in $\Omega$. We
will also need various notions of plurisubharmonic hulls.

\begin{definition}\label{def:hull} Let $\Omega \subset \Cn$, and let $K \relcomp
  \Omega$ be compact. We define the relative plurisubharmonic hull of $K$ in
  $\Omega$ by
  \begin{equation}\label{eq:defhull}
    \hull(K,\Omega) = \{ z \in \Omega \suchthat u(z) \le \sup_{x \in K} u(x)
    \text{ for every $u \in \plsh(\Omega)$} \}.
  \end{equation}
  In a similar way we define the smooth relative plurisubharmonic hull of $K$ in
  $\Omega$, denoted by $\hull^\infty(K,\Omega)$ where we only require that the
  inequality in~\eqref{eq:defhull} holds for functions in $\plsh(\Omega) \cap
  C^\infty(\Omega)$.
\end{definition}

\begin{remark}
  Note that $\hat K = \hull(K,\Cn) = \hull^\infty(K,\Cn)$. (Here, $\hat K$
  denotes the polynomial hull of $K$.) Furthermore, if $\Omega$ is pseudoconvex,
  then any plurisubharmonic function on $\Omega$ can be written as the limit of
  a decreasing sequence of smooth plurisubharmonic functions, and from this it
  follows that $\hull(K,\Omega) = \hull^\infty(K,\Omega)$. Also, if $\Omega$ is
  a neighbourhood of $\hat K$, then it can be checked that $\hat K =
  \hull(K,\Omega)$.
\end{remark}

\begin{proposition}\label{prop:suppT} Let $\Omega$ be a domain in $\Cn$ and
  assume that $T$ is a positive current of bidimension (1,1) with $\supp T
  \relcomp \Omega$. Assume that $\ddc T$ is negative on $\Omega \setminus K$ for
  some compact set $K \relcomp \Omega$. Then $\supp T \subset
  \hull^\infty(K,\Omega)$.
\end{proposition}

\begin{proof}
  Assume that $x \in \supp T \setminus \hull^\infty(K,\Omega)$. Then by
  definition of the hull, there is a $\phi \in \plsh(\Omega) \cap
  C^\infty(\Omega)$ such that $\phi(x) > 0$ and $\phi \le 0$ on a neighbourhood
  of $K$. Let $\psi = \max \{ \phi, 0 \}$ and define $u = \psi * \chi_\eps$
  where $\chi_\eps$ is a standard radial positive smothing kernel with support
  in $\mathbb{B}(0;\eps)$. Then $u$ is plurisubharmonic and smooth on
  $\Omega_\eps = \{ z \in \Omega : \dist(z,\partial\Omega) < \eps \}$. Choose
  $\eps$ so small that $\supp T \relcomp \Omega_\eps$ and that $u = 0$ on $K$.
  Then, $u \ge 0$, $u$ vanishes near $K$ and is strictly plurisubharmonic on a
  neighbourhood of $x$.  Hence
  \begin{equation*}
    0 < \langle T, \ddc \phi \rangle = \langle \ddc T, \phi \rangle \le 0,
  \end{equation*}
  which is a contradiction.
\end{proof}

\begin{proposition}\label{prop:T} Let $\Omega$ be a pseudoconvex domain in $\Cn$
  and let $\nu \in \E'(\Omega)$.  Assume that $\langle \nu,\phi \rangle \ge 0$
  for every $\phi \in \plsh(\Omega) \cap C^\infty(\Omega)$. Then there exists a
  positive current $T$ of bidimension (1,1) such that
  \begin{equation}\label{eq:ddc}
    \ddc T = \nu.
  \end{equation}
\end{proposition}

\begin{proof}
  Take an open set $U$ such that $\hull(\supp \nu, \Omega) \relcomp U
  \relcomp \Omega$.  Let $\Gamma = \{ \ddc T : T \ge 0 \text{ of
  bidimension (1,1)}, \supp T \subset \bar U \}$. Note that $\Gamma$
  is a closed cone in $\E'(\Omega)$ (in the weak topology).

  Assume that~\eqref{eq:ddc} has no solution $T$ with $\supp T \subset \bar U$.
  Then by Hahn--Banach's Theorem, $\nu$ is separated from the closed cone
  $\Gamma$, and there is a $\phi \in C^\infty(\Omega)$ and a constant $c$ such
  that
  \begin{equation}\label{eq:ineqc}
    \langle \nu, \phi \rangle < c \le \langle T, \ddc \phi \rangle
  \end{equation}
  for every $T \ge 0$ of bidimension (1,1) with support in $\bar U$.
  Since~\eqref{eq:ineqc} holds for $T$ replaced by $r T$ for every $r > 0$, we
  first note that $c \le 0$ (by taking $r$ sufficiently small), and secondly
  that $\langle T, \ddc \phi \rangle \ge 0$ for every $T \ge 0$ (by taking $r$
  sufficiently large). This implies that $\ddc \phi \ge 0$ on $U$, i.e.\ that
  $\phi$ is plurisubharmonic on $U$.  By modifying $\phi$ outside $\supp \nu$
  (using the assumption that $\hull(\supp \nu, \Omega) \relcomp U$) we can
  ensure that $\phi \in \plsh(\Omega) \cap C^\infty(\Omega)$, but then $\langle
  \nu, \phi \rangle < 0$ which contradicts the assumption on $\nu$.
\end{proof}

\begin{theorem}
    Let $\mu_1$ and $\mu_2$ be compactly supported in $\Omega$.  If $\mu_1
    \lepsh \mu_2$, and $\int \varphi\,d\mu_1 = \int \varphi\,d\mu_2$ for some
    strictly plurisubharmonic function $\varphi$, then $\mu_1 = \mu_2$.
\end{theorem}

    \begin{proof}
    Define a distribution $\nu$ by
    \[
        \langle \nu, \phi \rangle = \int \phi\,d\mu_1 - \int \phi\,d\mu_2
    \]
    for $\phi \in C_0^\infty$. Then $\nu$ is a compactly supported distribution
    that is non-negative on $\plsh(\Omega) \cap C^\infty(\Omega)$. By
    Proposition~\ref{prop:T}, we can write $\nu = \ddc T$ for some positive $T$
    with compact support.

    By assumption $0 = \langle \nu, \varphi \rangle = \langle \ddc T, \varphi
    \rangle = \langle T, \ddc \varphi \rangle$. Since $\varphi$ is strictly
    plurisubharmonic, $\ddc \varphi > 0$ on $\Omega$. This implies that $T = 0$,
    since $T$ is positive.
\end{proof}

\section{Edwards' theorem for other functionals}

We have already seen that it is possible to weaken some of the hypotheses in
Edwards' original formulation of his duality result. In this section, we will
explore another variation. We begin by noting that any cone $\mathcal{F}$ of
real-valued functions on~$X$ induces an ordering on the measures supported on
$X$:

\begin{definition}
    Let $\mu$ and $\nu$ be positive Radon measures supported on $X$. We say that
    $\mu \leF \nu$ if
    \[
        \int u\,d\nu \le \int u\,d\mu
    \]
    for all $u \in \mathcal{F}$.
\end{definition}

Note that, if $\mathcal{F}$ contains all constant functions, then $\mu \leF \nu$
implies that~$\mu$ and~$\nu$ have the same total mass. This induced ordering
allows us to reinterpret the original formulation of Edwards. Specifically, the
operator
\[
    Ig = \inf\{ g\,d\nu \suchthat \nu \in J^\mathcal{F}_x \},
\]
taking the infimum over the Jensen measures in $J^\mathcal{F}_x$, may be viewed
as taking the infimum over all measures $\nu$ satisfying $\nu \leF \delta_x$. On
the other hand, in the definition of
\[
    Sg = \sup\{ u(x) \suchthat u \in \mathcal{F}, u \le g \},
\]
$u(x)$ is the point evaluation $u \mapsto \int u\,d\delta_x$. This suggests
replacing point evaluation by another positive functional, i.e.~by integration
against any positive Radon measure~$\mu$ supported on $X$, yielding operators
\begin{align*}
   I_\mu g &= \inf \Big\{ \int g\,d\nu \suchthat \nu \leF \mu \Big\},
\shortintertext{and}
  S_\mu g &= \sup \Big\{ \int u\,d\mu \suchthat u \in \mathcal{F}, u\le g\Big\}
\end{align*}
for any bounded Borel function $g$ on $X$. Note that, in contrast to the usual
formulation of Edwards' machinery, this setup will not give us functions $S_\mu$
and $I_\mu$ on $X$ unless we make a choice of functionals $x \mapsto \mu_x$. At
the time of writing, we are not aware of any natural such choice other than $x
\mapsto \delta_x$.

Edwards' theorem continues to hold \emph{mutatis mutandis} for these operators
as well.

\begin{theorem}[Edwards' theorem for positive functionals]\label{edwards2}
    Let $X$ be a compact metric space, $H \subset C(X)$ be a subspace total
    outside its zeros, and suppose that $\mathcal{F}$ is $H$-admissible. Let
    $\mu$ be a positive Radon measure supported on $X$ and assume that there
    exists $C_\mu>0$ such that for each element $h \in H$, there exists $u\in
    \mathcal{F}$ with $u\leq h$ and
    \[
        \int u\,d\mu \geq -C_\mu |h|_\infty.
    \]
    Then for each non-decreasing sequence $h_i \in H$, we have
    $S_\mu(h_\infty)=I_\mu(h_\infty)$.
\end{theorem}

\begin{corollary}\label{edwards:ordering}
    Let $X$ be a compact metric space, suppose that $\mathcal{F}$ contains all
    constant functions, and let $\mu$ be a positive Radon measure. If $g$ is
    lower semicontinuous on $X$, then $S_\mu g = I_\mu g$.
\end{corollary}

This version of Edwards' theorem can be used to understand minimal elements in
the plurisubharmonic orderings of measures.

\begin{definition}
    Let $\mu$ be a positive Radon measure on $\bar\Omega$. We say that $\mu$ is
    \emph{minimal} with respect to the ordering $\lepshc$ induced by
    $\plsh(\Omega) \cap C(\bar\Omega)$ if
    \[
        \nu \lepshc \mu \quad\to\quad \nu = \mu,
    \]
    and similarly for the ordering $\lepshm$. Similarly, we say that $\mu$ is
    \emph{maximal} if $\mu \lepshc \nu$ implies that $\mu = \nu$.
\end{definition}

Bengtson~\cite{Bengtson} studies the maximal elements in $\lepsh$ and by using
Corollary~\ref{edwards:ordering}, we can characterize the minimal elements in
our plurisubharmonic orderings of measures.

\begin{theorem}\label{thm:minimal}
    Let $\Omega \subset \Cn$ be a B-regular domain. A positive Radon
    measure~$\mu$ on~$\bar\Omega$ is minimal with respect to $\lepshc$ if and
    only if $\supp \mu \subset \partial\Omega$.
\end{theorem}

\begin{proof}
    Assume that $\mu$ is minimal with respect to $\lepshc$. Let $g \in
    C(\bar\Omega)$. By Corollary~\ref{edwards:ordering},
    \[
        S_\mu g =
        \sup \Big\{ \int u\,d\mu \suchthat u \in \plsh(\Omega) \cap
            C(\bar\Omega), u\le g\Big\}
        = I_\mu g = \int g\,d\mu.
    \]
    On the other hand,
    \[
    S_\mu g = \int (Pg)\,d\mu,
    \]
    where $Pg(z) = \sup \{ u(z) : u \in \plsh(\Omega) \cap C(\bar\Omega) \}$ is
    the Perron--Bremermann envelope of $g$. (Since $g \in C(\bar\Omega)$, it
    follows that $(Pg)^* = Pg$ is plurisubharmonic and continuous on $\Omega$,
    and from the characterization of Jensen measures in~\cite{Wikstrom}, we see
    that it makes no difference whether the supremum defining $Pg$ is taken over
    $\plsh(\Omega)$ or $\plsh(\Omega) \cap C(\bar\Omega)$.)

    Hence, for every $g \in C(\bar\Omega)$, $\int (Pg)\,d\mu = \int g\,d\mu$,
    but since $Pg \le g$, we can conclude that $Pg = g$ almost everywhere
    $[\mu]$. By continuity, $Pg = g$ on $\supp \mu$.

    In particular, we can choose a continuous function $g \le 0$, such that
    $g^{-1}(0) = \supp \mu$. Then $Pg \le 0$, $Pg = g$ on $\supp \mu$ and by the
    maximum principle, $\supp \mu \subset \partial\Omega$.

    Conversely, assume that $\supp \mu \subset \partial \Omega$ and $\nu \lepshc
    \mu$. Let $\phi \in \plsh(\Omega) \cap C(\bar\Omega)$ be a bounded
    plurisubharmonic exhaustion function for $\Omega$, i.e.
    $\phi|_{\partial\Omega} = 0$ and $\phi < 0$ on~$\Omega$.

    Then
    \[
        0 \ge \int \phi\,d\nu \ge \int \phi\,d\mu = 0
    \]
    In particular $\phi = 0$ almost everywhere~$[\nu]$, so $\supp \nu \subset
    \partial\Omega$. Proposition~\ref{prop:equal_mass} shows that $\nu = \mu$,
    i.e. $\mu$ is minimal.
\end{proof}

In Bengtson's ordering, there are no minimal measures by a definition similar to
above since $\frac{1}{2}\mu \lepsh \mu$ holds for any positive Radon measure $\mu$.
One way to circumvent this issue is to instead consider the weaker minimality
property
\[
    \nu \lepsh \mu \quad\to\quad \nu = k_\nu\mu
\]
for some $k_\nu \in [0,1]$. Even with this weaker definition, there are no
minimal elements.

\begin{theorem}\label{thm:minimal2}
    Let $\Omega \subset \Cn$ be a hyperconvex domain, and suppose that $\mu$ is
    positive Radon measure on $\Omega$ satisfying the minimality property above.
    Then $\mu(\Omega)=0$.
\end{theorem}

\begin{proof}
    Pick two precompact, disjoint open sets $A, B \subset \Omega$ and let $g \in
    C(\bar \Omega)$ satisfy
    \[
        g\leq 0, \qquad g\big|_{A}=g\big|_{\bar \Omega}=0, \qquad g\big|_{B} = -1.
    \]
    By Theorem~\ref{thm:edwhypcon2}, $Pg=\sup\big\{u(x) \suchthat u\in
    \E_0(\Omega), u \leq g\big\}$ is plurisubharmonic and continuous since
    \begin{align*}
        Pg  &\geq \sup\big\{u(x) \suchthat u\in \E_0(\Omega) \cap C(\bar \Omega), u \leq g\big\} \\
            &= \sup\big\{u(x) \suchthat u\in \psh^0(\Omega) , u \leq g\big\} \\
            &\geq Pg.
    \end{align*}
    Applying Theorem~\ref{edwards2}, we have
    \[
        \int (Pg)\,d\mu=S_\mu g= I_\mu g =
        \inf_{k_\nu\in[0,1]}\Big\{k_\nu \int g\,d\mu\Big\} = \int g\,d\mu,
    \]
    and as in the proof of Theorem~\ref{thm:minimal}, it follows that $Pg = g$ on $\supp \mu$. By the maximum principle, $P(g)\big|_A < 0$, which implies that $\mu(A)=0$. Since $A$ was arbitrary, the conclusion of the theorem follows.
\end{proof}

\bibliographystyle{amsplain}

\end{document}